\documentclass[10pt,a4paper]{article}
\usepackage{amsfonts}
\usepackage[latin9]{inputenc}
\usepackage{amsmath}
\usepackage{amssymb}
\usepackage{url}
\usepackage{graphicx}
\usepackage[pdfstartview=FitH]{hyperref}
\usepackage{amsthm}
\usepackage{authblk}
\makeatletter
\begin{document}
\newtheorem{theorem}{Theorem}[section]
\newtheorem{lemma}[theorem]{Lemma}
\newtheorem{definition}[theorem]{Definition}
\newtheorem{claim}[theorem]{Claim}
\newtheorem{example}[theorem]{Example}
\newtheorem{remark}[theorem]{Remark}
\newtheorem{proposition}[theorem]{Proposition}
\newtheorem{corollary}[theorem]{Corollary}

\title{Property A and the existence of a Markov process with a trivial Poisson boundary}
\author{Izhar Oppenheim}
\affil{Department of Mathematics\\
 The Ohio State University  \\
 Columbus, OH 43210, USA \\
E-mail: izharo@gmail.com}

\maketitle

\begin{abstract}
This note make the observation that property A for a space is equivalent to the existence of a Markov process on the space which has a (uniformly) trivial Poisson boundary. 
\end{abstract}

\section{Introduction} 
\label{intro}

In \cite{Yu}, Yu introduced property A which can be viewed as a non-equivariant analogue of amenability. To explain what "non-equivariant" means, we recall Reiter's necessary and sufficient condition for amenability:
\begin{theorem}
\label{Reiter}
A countable group $G$ is amenable if and only if there exists a sequence of non degenerate probability measures $\lbrace \mu_n \rbrace_{n \in \mathbb{N}}$ on $G$ such that  
$$ \forall g \in G, \lim_{n \rightarrow \infty} \Vert  g \mu_n - \mu_n \Vert_1 = 0.$$
\end{theorem}
One can think of the sequence of measures in the above theorem as a sequence of maps $\phi_n : G \rightarrow Prob(G)$ that is defined as $\phi_n (g, h) = \mu_n (g^{-1} h), \forall g \in G$ (we denote $\phi_n (g) (h) = \phi (g,h)$). Compare this condition to the condition of Higson and Roe \cite{HR} for property A:
\begin{theorem}
\label{HR}
A finitely generated group has property A if and only if there exists a sequence of maps $\phi_n : G \rightarrow Prob(G)$ with the following properties:
\begin{enumerate}
\item For every $n$ there is a finite subset $F \subset G$ such that for every $g \in G$ we have $g^{-1} Supp(\phi_n (g,. )) \subset F$.
\item For every $g \in G$ we have
$$\lim_{n \rightarrow \infty} \sup_{h \in G} \Vert \phi_n (h,.) - \phi_n (hg,.) \Vert_1 = 0.$$

\end{enumerate}
\end{theorem}
One can think the the two theorems above are not completely analogues, because in the condition for amenability in theorem \ref{Reiter}, $\mu_n$ does not have to be of finite support. This can be amended in two ways: first, without any loss of generality, one can add the assumption of finite support to theorem \ref{Reiter} (this is due to the fact that one can approximate any $\mu_n$ by $\mu_n'$ with finite support). The second way, which we will use, because it is more convenient for the purpose of this article, is to relax the finite support condition of the map $\phi_n$ in theorem \ref{HR}. We shall see below that the finite support condition can be replaced by the condition: \\
\textit{
For every $n$ and for every $\delta >0$ there is a finite set $F_\delta \subset G$ such that for every $g \in G$ we have 
$$ \sum_{h \in F_\delta} \phi_n (g, h) > 1 - \delta.$$
}
\\ \\
In response to a conjuncture by Furstenberg, it was proven in \cite{Rosen} and in \cite{KV} that the sequence of measures in the condition for amenability can be given as a sequence of convolutions of a single measure: 
\begin{theorem}
A countable group $G$ is amenable if and only if there exist a probability measure $\mu$ on $G$ such that 
$$ \forall g \in G, \lim_{n \rightarrow \infty} \Vert  g \mu^{*n} - \mu^{*n} \Vert_1 = 0.$$
\end{theorem}
It is also shown in \cite{Rosen}, \cite{KV} that this condition for amenability is equivalent to the triviality of the Poisson boundary of $(G,\mu)$.  \\
The easy observation made in this note is that the technique used in \cite{KV} needs very little adaptation to the case of property A, namely we show that:
\begin{theorem}
\label{Poisson-group}
A finitely generated group has property A if and only if there exists a transition probability $P$ with the state space $G$ with the following properties:
\begin{enumerate}
\item For every $\delta >0$ there is a finite set $F_\delta \subset G$ such that for every $g \in G$ we have 
$$ \sum_{h \in F_\delta} P  (g,gh) > 1 - \delta.$$
\item For every $g \in G$ we have
$$\lim_{n \rightarrow \infty} \sup_{h \in G} \Vert P^n (h,.) - P^n (hg,.) \Vert_1 = 0.$$
\end{enumerate}  
\end{theorem}
The second condition above implies the triviality of the Poisson boundary of the Markov chain defined by $P$ and any initial probability of $G$, but it is not equivalent to it. One can think of this condition as a metric uniform version of the condition for triviality of the Poisson boundary (the sufficient condition for triviality of the Poisson boundary of the Markov chain for any initial probability, requires a convergence of the sequence of products that is much more tame). \\ \\
The above theorem is stated for finitely generated groups, but applies in a more general metric setting of bounded geometry. The structure of this note is as following - in section 2 we will give the necessary background on property A (for discrete metric spaces of bounded geometry) and on the Poisson boundary of a Markov process (this background will be far from complete, since a complete background on any of the mentioned topics is far beyond the scope of this note. Regarding property A, the interested reader can consult \cite{YuNo} for a gentle exposition on this subject and \cite{Will} for an extensive survey. Regarding Possion boundary of Markov processes, the interested reader can consult \cite{Kaim} and the references mentioned in it). In section 3 we will prove theorem \ref{Poisson-group} for the general case of metric spaces with bounded geometry. \\ \\

The author wants to thank Uri Bader and Amos Nevo for introducing him to the Poisson boundary and to Piotr Nowak for reading an early draft of this article. 

\section{Background}

\subsection{Markov chains and the Poisson boundary}
\subsubsection{Measures and convolutions}
Let $X$ be a countable set, and let $Prob(X)$ be the space of functions $\mu : X \rightarrow \mathbb{R}_{\geq 0}$ such that $\sum_x \mu (x) =1$. The $L^1$ metric on $Prob(X)$ is defined as usual to be 
$$\Vert \mu - \nu \Vert_1 = \sum_x \vert \mu (x)-\nu (x) \vert,$$
(since we are dealing with a countable space, this metric is equivalent to the total variation metric). \\
A transition probability on the state space $X$ is a map $\phi : X \rightarrow Prob(X)$ so $\phi (x) \in Prob(X)$ and we shall use the notation $ \phi (x)(y) = \phi (x,y)$. \\
Given a measure $\mu \in Prob(X)$ and a transition probability $\phi$ on $X$, the convolution $\mu * \phi \in Prob(X)$ is defined as
$$ (\mu * \phi) (y) = \sum_x \mu (x) \phi (x,y).$$
Given two transition probabilities $\phi, \psi$ on $X$, the convolution $\phi * \psi$ is a transition probability on $X$ defined as
$$ (\phi * \psi ) (x,y) = \sum_z \phi (x,z) \psi (z,y).$$
For a transition probability $P$ we shall use the notation $P^n = P*...*P$ and $P^n (x,.)$ is $\delta_x*P^n$. Below is an easy proposition about the interplay between the convolution and the $L^1$ metric (the proof is given for the sake of completeness).

\begin{proposition}
\label{convolution}
Let $\mu ,\nu \in Prob(X)$ and be $\phi$ a transition probability on $X$, then
\begin{enumerate}
\item Convolution (from the right) with $\phi$ is a non expending map, i.e.,
$$ \Vert \mu *\phi - \nu *\phi \Vert_1 \leq \Vert \mu - \nu \Vert_1.$$
\item If there exists $\varepsilon >0$ and $x_0 \in X$ such that for all $x \in Supp(\mu ) \cup Supp(\nu )$ we have $\Vert \phi (x,.) - \phi (x_0,.) \Vert_1 < \varepsilon$ then 
$$ \Vert \mu *\phi - \nu *\phi \Vert_1 < 2 \varepsilon.$$
\end{enumerate}
\end{proposition}

\begin{proof}
\begin{enumerate}
\item For every $\mu ,\nu \in Prob(X)$ we have
$$ \Vert \mu *\phi - \nu*\phi \Vert_1 = \sum_y \vert \sum_x \mu (x) \phi (x,y) -  \sum_x \nu(x) \phi (x,y) \vert \leq$$
$$ \leq  \sum_y  \sum_x \vert (\mu (x)-\nu (x))\vert \phi (x,y)  = \sum_x \vert \mu (x)-\nu (x)\vert = \Vert \mu  - \nu \Vert_1.$$
\item Notice that since $\mu ,\nu \in Prob(X)$ we have for any $y$ that
$$ \phi (x_0,y) = \sum_x \mu (x) \phi (x_0,y) = \sum_x \nu (x) \phi (x_0,y).$$
This implies 
$$ \Vert \mu *\phi - \nu *\phi \Vert_1 = \sum_y \vert \sum_x \mu (x) \phi (x,y) -  \sum_x \nu (x) \phi (x,y) \vert = $$
$$ = \sum_y \vert \sum_x \mu (x) \phi (x,y) - \sum_x \mu (x) \phi (x_0,y)  + \sum_x \nu (x) \phi (x_0,y) -   \sum_x \nu (x) \phi (x,y) \vert \leq $$
$$ \leq \sum_y \vert \sum_x \mu (x) (\phi (x,y) - \phi (x_0,y)) \vert + \sum_y \vert \sum_x \nu (x) (\phi (x,y) - \phi (x_0,y) )\vert \leq $$
$$ \leq \sum_{x \in Supp(\mu )} \mu (x) \sum_y  \vert  \phi (x,y) - \phi (x_0,y) \vert +  \sum_{x \in Supp(\nu )} \nu (x) \sum_y  \vert  \phi (x,y) - \phi (x_0,y) \vert \leq $$
$$<  \sum_{x \in Supp(\mu )} \mu (x) \varepsilon +  \sum_{x \in Supp(\nu )} \nu (x) \varepsilon = 2 \varepsilon.$$ 
\end{enumerate}
\end{proof}

\subsubsection{The Poisson boundary of a Markov chain}
A triple $(X,\mu,P)$ where $X$ is a countable state space, $\mu \in Prob(X)$ and $P$ is a transition probability on $X$ is called a (time homogeneous) Markov chain. For a Markov chain $(X,\mu,P)$, $\mu$ is called the initial probability of the chain.  \\
Given a countable state space $X$, define the measure space $X^\mathbb{N}$ (the sequences of elements of $X$) with the $\sigma$-algebra $\mathcal{A}$ generated by the cylinder sets:
$$[x_1,...,x_n] = \lbrace \omega = (y_1,...) \in X^\mathbb{N} : y_1=x_1,...,y_n=x_n \rbrace.$$
A Markov chain $(X,\mu,P)$ defines a probability measure $P_\mu$ on $( X^\mathbb{N}, \mathcal{A} )$ given as 
$$P_\mu ([x_1,...,x_n]) =  \mu (x_1) P(x_1,x_2) ... P(x_{n-1},x_n).$$
Define the operator $T: X^\mathbb{N} \rightarrow X^\mathbb{N} $ as 
$$ T (x_1,x_2,...)=(x_2,x_3,...).$$
Then the Poisson boundary of $(X,\mu,P)$ is defined as the ergodic components of $T$ under the measure $P_{\mu}$. We say that the Poisson boundary is trivial if there is only one ergodic component. In \cite{Der}, Derriennic proved the following (see also \cite{Kaim}):
\begin{theorem}
\label{triviality-condition}
Given $X$ and a transition probability $P$, the Poisson boundary of $(X,\mu,P)$ is trivial for all $\mu \in Prob(X)$ if and only if for every $x,y \in X$ we have
$$ \lim_{n \rightarrow \infty} \frac{1}{n} \Vert \sum_{i=1}^n P^i (x,.) - \sum_{i=1}^n P^i (y,.) \Vert_1 = 0.$$
\end{theorem}
 
When $X=G$ is a countable group, the Poisson boundary is defined with respect to a transition probability $P$ which is invariant under the group action, i.e., 
$$ \forall h,g \in G, P(g,gh)=P(e,h).$$
Rosenblatt \cite{Rosen} and Kaimanovich-Vershik \cite{KV} proved the following theorem characterizing amenable groups in terms of the Poisson boundary:
\begin{theorem}
A countable group $G$ is amenable if and only if there is an invariant transition probability $P$ on $G$ such that the Poisson boundary of $(G,\mu,P)$ is trivial for all $\mu \in Prob(G)$. 
\end{theorem} 

\subsection{Bounded geometry and Property A}

In this section will give the basic definitions regarding metric spaces with bounded geometry and Property A. Throughout this entire paper we assume that our metric space $(X,d)$ is discrete and countable.

\begin{definition}
A discrete metric space $(X,d)$ is said to have bounded geometry if for every $C>0$ there is a number $M(C)$ such that for every $x \in X$ we have
$$ \vert B(x,C) \vert \leq M(C).$$ 
\end{definition}

The following example is was of the main motivations to study discrete metric spaces with bounded geometry.

\begin{example}
Let $G$ be a finitely generated group with a generating set $S$, then it is obvious the vertex set of the Cayley graph of $G$ with respect to $S$, with the graph metric on it, is a discrete metric space with bounded geometry. 
\end{example}

\begin{definition}
A discrete metric space $(X,d)$ is said to have Property A if there is a collection $\lbrace A_x^n \rbrace_{x \in X, n \in \mathbb{N}} $ of finite subsets of $X \times \mathbb{N}$ such that the following holds:
\begin{enumerate}
\item For every $n \in \mathbb{N}$ there is a number $R_n$ such that for every $x \in X$ we have
$$ A_x^n \subset B(x,R_n) \times \mathbb{N}.$$
\item For every $K>0$ we have that
$$ \lim_{n \rightarrow \infty} \sup_{d(x,y) < K}  \dfrac{\vert A_x^n \bigtriangleup A_y^n \vert}{\vert A_x^n \cap A_y^n \vert} = 0.$$
\end{enumerate}
\end{definition}

In \cite{HR} Higson and Roe gave a characterization of property A that is analogues to Reiter's condition for amenability. Namely they proved the following (when this theorem is applied to the example of finitely generated groups, one gets theorem \ref{HR}):

\begin{theorem}
\label{HR-metric}
Let $(X,d)$ be a countable discrete metric space with bounded geometry, then $(X,d)$ has property A if and only if there is a sequence of  transition probabilities $\phi_n: X \rightarrow Prob(X)$ such that the following holds:
\begin{enumerate}
\item For every $n$ there is a number $R_n$ such that for every $x \in X$ we have 
$$Supp (\phi_n (x,.) ) \subset B(x,R_n).$$
\item For every $K>0$ we have
$$\lim_{n \rightarrow \infty} \sup_{d(x,y) <K} \Vert \phi_n (x,.) - \phi_n (y,.) \Vert_1 =0.$$
\end{enumerate}
\end{theorem} 

In the next proposition, we show that the first condition in the above theorem can be weakened:

\begin{proposition}
\label{HR-refined}
Let $(X,d)$ be a countable discrete metric space with bounded geometry, then $(X,d)$ has property A if and only if there is a sequence of  transition probabilities $\phi_n: X \rightarrow Prob(X)$ such that the following holds:
\begin{enumerate}
\item For every $n$ and for every $\delta >0$ there is there is a number $R_{\delta, n}$ such that for every $x \in X$ we have 
$$\sum_{y \in B(x, R_{\delta, n})} \phi_n (x,y) > 1 - \delta.$$
\item For every $K>0$ we have
$$\lim_{n \rightarrow \infty} \sup_{d(x,y) <K} \Vert \phi_n (x,.) - \phi_n (y,.) \Vert_1 =0.$$
\end{enumerate}
\end{proposition}

\begin{proof}
The first direction is obvious from theorem \ref{HR-metric} - for every $n$ and every $\delta>0$ choose $R_{\delta, n} = R_n$. Conversely, assume there is a sequence of  transition probabilities $\phi_n: X \rightarrow Prob(X)$ with the conditions stated above and define a new sequence $\phi_n'$ as follows: 
$$ \phi_n' (x,y) = \begin{cases}
\dfrac{1}{\sum_{y \in B(x, R_{\frac{1}{n}, n})} \phi_n (x,y)} \phi_n (x,y) & y \in B(x, R_{\frac{1}{n}, n}) \\
0 & d(x,y) \geq R_{\frac{1}{n}, n}
\end{cases} .$$
Then $\phi_n'$ is a sequence of transition probability such that for every $x \in X$ we have
$$Supp (\phi_n' (x,.)) \subseteq B(x, R_{\frac{1}{n}, n}),$$
and also for every $n >1$ and every $x \in X$ we have that
$$ \Vert \phi_n (x,.) - \phi_n' (x,.) \Vert_1 = \sum_{y \in B(x, R_{\frac{1}{n}, n})} \phi_n (x,y) ( \dfrac{1}{\sum_{y \in B(x, R_{\frac{1}{n}, n})} \phi_n (x,y)} -1 ) +$$
$$+ \sum_{y, d(x,y) \geq R_{\frac{1}{n}, n}} \phi_n (x,y)  \leq \dfrac{1}{1 - \frac{1}{n}}- 1 + \dfrac{1}{n} \leq \dfrac{3}{n}.$$
Thus, for every $x,y \in X$ and every $n>1$ we have that 
$$ \Vert \phi_n' (x,.) - \phi_n' (y,.) \Vert_1 \leq \Vert \phi_n' (x,.) - \phi_n (x,.) \Vert_1 +$$
$$+ \Vert \phi_n (x,.) - \phi_n (y,.) \Vert_1 + \Vert \phi_n (y,.) - \phi_n' (y,.) \Vert_1 \leq  \Vert \phi_n (x,.) - \phi_n (y,.) \Vert_1 + \dfrac{6}{n},$$
and therefore
$$\lim_{n \rightarrow \infty} \sup_{d(x,y) <K} \Vert \phi_n' (x,.) - \phi_n' (y,.) \Vert_1 \leq \lim_{n \rightarrow \infty} \left( \sup_{d(x,y) <K} \Vert \phi_n (x,.) - \phi_n (y,.) \Vert_1 + \dfrac{6}{n} \right) = 0,$$
and the conditions of theorem \ref{HR-metric} hold for $\phi_n'$.
\end{proof}

\section{Property A as uniform triviality of the Poisson boundary}

\begin{theorem}
\label{Poisson-metric}
Let $(X,d)$ be a countable, discrete metric space with a bounded geometry. Then $(X,d)$ has property A if and only if there exists a transition probability $P$ on the state space $X$ with the following properties:
\begin{enumerate}
\item For every $\delta >0$ there is some $R_\delta$  such that for every $x \in X$ we have 
$$ \sum_{y \in B(x,R_\delta)} P  (x,y) > 1 - \delta.$$
\item For every $K>0$ we have that
$$\lim_{n \rightarrow \infty} \sup_{d(x,y)<K}  \Vert P^n (x,.) - P^n (y,.) \Vert_1 = 0.$$
\end{enumerate}  
\end{theorem}
\begin{proof}
Let $(X,d)$ be a countable, discrete metric space with a bounded geometry and assume that $(X,d)$ has property A, therefore by theorem \ref{HR-metric} there is a sequence of maps $\phi_n: X \rightarrow Prob(X)$ such that:
\begin{enumerate}
\item  For every $n$ there is $R_n$ such that for all $x \in X$ we have that $Supp (\phi_n (x,.) ) \subset B(x,R_n)$.
\item For every $K>0$ we have that
$$\lim_{n \rightarrow \infty} \sup_{d(x,y)<K} \Vert \phi_n (x,.) - \phi_n (y,.) \Vert_1  = 0.$$
\end{enumerate}
Choose two sequences of positive real numbers $\lbrace t_i \rbrace, \lbrace \varepsilon_i \rbrace$ such that $\sum_i t_i =1$ and $\lim_{i \rightarrow \infty} \varepsilon_i = 0$. Let $\lbrace n_i \rbrace$ be an increasing sequence of natural numbers such that
$$(t_1+...+t_{i-1})^{n_i} < \varepsilon_i.$$
Define a subsequence of $\phi_i$ which we will also denote $\phi_i$ inductively: \\
let $\phi_1$ be the map such that for every $x,y \in X$ with $d(x,y)<1$ we have 
$$\Vert \phi_1 (x,.) - \phi_1 (y,.) \Vert_1 < \varepsilon_1.$$
and let $R_1$ be a number such that $Supp (\phi_1 (x,.)) \subset B(x,R_1)$ for every $x \in X$. \\
Define $\phi_i$ to be the transition probability such that for every $x,y \in X$ with $d(x,y) < n_i R_{i-1}$ we have
$$\Vert \phi_i (x,.) - \phi_i (y,.) \Vert_1 < \varepsilon_i,$$
and let $R_i$ be a number such that $Supp (\phi_i (x,.)) \subset B(x,R_i)$ for every $x \in X$ and without loss of generality we choose $R_i > \max \lbrace R_{i-1},i \rbrace$. \\
We shall show that $P = \sum_i t_i \phi_i$ is a transition probability with the properties stated in the theorem. First note that $P$ is well defined, since for every $x,y \in X$, $P(x,y) =  \sum_i t_i \phi_i (x,y)$ is a series with non negative terms that is bounded from above by $\sum_i t_i = 1$ (since $\forall i, \phi_i (x,y) \leq 1$) and therefore the series $\sum_i t_i \phi_i (x,y)$ is convergent. Also, for every $x \in X$ we have $\sum_{y} P(x,y) = 1$, since we can change the order of summation due to the fact that the series is absolutely convergent. \\
Next, note that for every $\delta >0$ there is some $i_0$ such that $\sum_{i=i_0}^\infty t_i < \delta$. Choose $R_\delta = R_{i_0}$. For every $x \in X$ we have that $Supp (\phi_i (x,.)) \subset B(x,R_i)$, and since we choose $R_i$ to be monotone increasing, we have
$$\forall i < i_0, \forall x,y \in X, d(x,y) \geq R_{i_0} \Rightarrow \phi_i (x, y) = 0.$$
Therefore   
$$ \sum_{y, d(x,y) \geq R_{i_0}} P  (x,y) = \sum_{y, d(x,y) \geq R_{i_0}} \sum_{i=i_0}^\infty t_i \phi_i (x, y) \leq \sum_{i=i_0}^\infty t_i < \delta.$$
Thus, we got that 
$$ \sum_{y \in B(x,R_{i_0})} P  (x,y) > 1 - \delta.$$
To prove the second condition we shall show that for every $x,y \in X$ with $d(x,y) < R_{i-1}$ we have  
$$\Vert P^{n_i} (x,.) -  P^{n_i} (y,.) \Vert_1 \leq 4 \varepsilon_i,$$
(it is sufficient to work with the subsequence $n_i$ because from \ref{convolution} we get that for every $n>n_i$ we have
$$\Vert P^{n} (x,.) -  P^{n} (y,.) \Vert_1 \leq \Vert P^{n_i} (x,.) -  P^{n_i} (y,.) \Vert_1$$
). \\
By definition we have that
$$P^{n_i} (.,.)  = \sum_{(k_1,...,k_{n_i}) \in \mathbb{N}^{n_i}} t_{k_1} ... t_{k_{n_i}} \phi_{k_1} * ... * \phi_{k_{n_i}}.$$
Define
$$A_i = \lbrace (k_1,...,k_{n_i}) \in \mathbb{N}^{n_i} : k_j < i, \forall j \rbrace$$
$$B_i = \mathbb{N}^{n_i} \setminus A_i.$$
This yields 
$$ P^{n_i} (.,.)  = \sum_{(k_1,...,k_{n_i}) \in A_i} t_{k_1} ... t_{k_{n_i}} \phi_{k_1} * ... * \phi_{k_{n_i}} + \sum_{(k_1,...,k_{n_i}) \in B_i} t_{k_1} ... t_{k_{n_i}} \phi_{k_1} * ... * \phi_{k_{n_i}}.$$
Note that for every $x \in X$
$$ \Vert \sum_{(k_1,...,k_{n_i}) \in A_i} t_{k_1} ... t_{k_{n_i}} \phi_{k_1} * ... * \phi_{k_{n_i}} (x,.) \Vert_1 \leq$$
$$\leq \sum_{(k_1,...,k_{n_i}) \in A_i} t_{k_1} ... t_{k_{n_i}}  = (t_1+..+t_{i-1})^{n_i} < \varepsilon_i.$$
It follows that for every $x,y \in X$ we have
$$\Vert P(x,.) - P(y,.) \Vert_1 < 2 \varepsilon_i + $$
$$ + \Vert \sum_{(k_1,...,k_{n_i}) \in B_i} t_{k_1} ... t_{k_{n_i}} \left( ( \phi_{k_1} * ... * \phi_{k_{n_i}} ) (x,.) - ( \phi_{k_1} * ... * \phi_{k_{n_i}} ) (y,.) \right) \Vert_1 \leq$$
$$ \leq 2 \varepsilon_i +  \sum_{(k_1,...,k_{n_i}) \in B_i} t_{k_1} ... t_{k_{n_i}} \Vert ( \phi_{k_1} * ... * \phi_{k_{n_i}} ) (x,.) - ( \phi_{k_1} * ... * \phi_{k_{n_i}} ) (y,.)  \Vert_1.$$
Thus, in order to show 
$$\Vert P(x,.) - P(y,.) \Vert_1 < 4 \varepsilon_i,$$
it is sufficient to show that for every $(k_1,...,k_{n_i}) \in B_i$ we have 
$$\Vert ( \phi_{k_1} * ... * \phi_{k_{n_i}} ) (x,.) - ( \phi_{k_1} * ... * \phi_{k_{n_i}} ) (y,.)  \Vert_1 \leq 2 \varepsilon_i.$$
Let $(k_1,...,k_{n_i}) \in B_i$ and let $j$ be the largest index such that $k_j < i$ (so $1 \leq j < n_i$). Denote 
$$\mu  = \phi_{k_1} * ... * \phi_{k_j} (x,.) \in Prob(X),$$
$$\nu = \phi_{k_1} * ... * \phi_{k_j} (y,.) \in Prob(X).$$
Then we need to show 
$$ \Vert \mu * \phi_{k_{j+1}}*...*\phi_{k_{n_i}}  - \nu * \phi_{k_{j+1}}*...*\phi_{k_{n_i}} \Vert_1 < 2 \varepsilon_i,$$
and by proposition \ref{convolution} it is enough to show
$$ \Vert \mu * \phi_{k_{j+1}} - \nu * \phi_{k_{j+1}} \Vert_1 < 2 \varepsilon_i.$$
Since $k_1 ,.., k_j < i$ we have that $Supp(\mu ) \subset B(x, j R_{i-1})$ and $Supp(\nu ) \subset B(y, j R_{i-1})$. Also, since $d(x,y) < R_{i-1}$ we have that $Supp(\nu ) \subset B(x, (j+1) R_{i-1}) \subset B(x, n_i R_{i-1})$ and obviously $Supp(\mu ) \subset B(x, n_i R_{i-1})$. Since $k_{j+1} \geq i$ we have from the choice of $R_{k_{j+1}} \geq R_i \geq n_i R_{i-1}$ that for all $x' \in B(x, n_i R_{i-1})$ we have
$$ \Vert \phi_{k_{j+1}} (x,.) - \phi_{k_{j+1}} (x',.) \Vert_1 < \varepsilon_i.$$
Applying proposition \ref{convolution} we get that 
$$ \Vert \mu * \phi_{k_{j+1}} - \nu * \phi_{k_{j+1}} \Vert_1 < 2 \varepsilon_i,$$
and the proof of the first direction is complete. \\ \\

The other direction follows from proposition \ref{HR-refined} - if there is a transition probability $P$ with the above properties, define $\phi_n: X \rightarrow Prob (X)$ as $\phi_n (x,.) = P^n (x,.)$.
We need only to show that for every $n$ and every $\delta >0$ there is some $R_{\delta,n}$ such that for every $x \in X$ we have
$$ \sum_{y \in B(x,R_{\delta,n})} P^n  (x,y) > 1 - \delta.$$
If $n=1$ by the conditions of the theorem for every $\delta>0$ we have $R_{\delta,1}$ that meets the above requirement. Fix $\delta > 0$, $n>1$ and choose $R_{\delta,n} = n R_{\frac{\delta}{n},1}$. 
Indeed observed that for every $x \in X$ we have 
$$ \sum_{y, d(x,y) \geq R_{\delta,n}} P^n  (x,y) \leq $$
$$ \leq \sum_{i=1}^{n} \sum_{z_1 \in X} P^{i-1} (x,z_1) \sum_{z_2 \in X, d(z_1,z_2) \geq R_{\frac{\delta}{n},1}} P(z_1,z_2) \sum_{y \in X} P^{n-i} (z_2,y) \leq  \delta.$$

\end{proof}

When the theorem above is applied to the case of finitely generated groups one get theorem \ref{Poisson-group} stated in the introduction. 

\begin{remark}
The sufficient condition for the triviality of the Poisson boundary of $X,P$ of theorem \ref{triviality-condition} can be read in the following way: the Poisson boundary of $X,P$ is trivial for any initial probability if and only if the sequence of transition probabilities defined as 
$$\phi_n (x,.) = \frac{1}{n} \left( P(x,.) + ... + P^n (x,.) \right),$$
satisfies the condition
$$\forall x,y \in X, \lim_{n \rightarrow \infty} \Vert \phi_n (x,.) - \phi_n (y,.) \Vert_1 = 0.$$
The above condition does not require $X$ to be a metric space and therefore does not take into account any interplay between the metric and the measure. Comparing the above condition to the conditions given in theorem \ref{Poisson-metric}, property A is equivalent to stronger conditions:
\begin{enumerate}
\item The needed limit is of convolutions and not an average of convolutions (see \cite{Rosen} to the difference between "ergodic by convolutions" and "mixing by convolutions").
\item The limit has uniformity with respect to the metric. 
\end{enumerate} 
Therefore one can think of the conditions of theorem \ref{Poisson-metric} as "uniform" triviality of the Poisson boundary. One should note, that without the uniformity in the condition, it becomes trivial as stated in the next proposition.
\end{remark}

\begin{proposition}
For every metric space $(X,d)$ there is a transition probability $P$ with 
\begin{enumerate}
\item For every $\delta >0$ there is some $R_\delta$  such that for every $x \in X$ we have 
$$ \sum_{y \in B(x,R_\delta)} P  (x,y) > 1 - \delta.$$
\item For every $x,y \in X$ we have that 
$$\lim_{n \rightarrow \infty} \Vert P^n (x,.) - P^n (y,.) \Vert_1 = 0.$$
\end{enumerate}  
\end{proposition}

\begin{proof}
Fix some $x_0 \in X$ and define 
$$P(x,.) = \begin{cases}
\delta_{x_0} & x=x_0 \\
\frac{1}{1+d(x,x_0)} \delta_{x_0} + \frac{d(x,x_0)}{1+d(x,x_0)} \delta_x & x \neq x_0
\end{cases}.$$
Then for every $\delta >0$ take $R_\delta = \frac{1}{\delta} $ and check that
$$ \sum_{y \in B(x,R_\delta)} P  (x,y) > 1 - \delta.$$
Also, for every $x \neq x_0$ we have
$$P^n (x,.) =(1-(\frac{d(x,x_0)}{1+d(x,x_0)})^n) \delta_{x_0} +(\frac{d(x,x_0)}{1+d(x,x_0)})^n \delta_x,$$
and therefore we have for every $x,y \in X$ that 
$$ \Vert P^n (x,.) - P^n (y,.) \Vert_1 \leq 2 (\frac{d(x,x_0)}{1+d(x,x_0)})^n +2 (\frac{d(y,x_0)}{1+d(y,x_0)})^n,$$
and the proposition is proved.
\end{proof}

\bibliographystyle{plain}
\bibliography{bibl}

\end{document}